\documentclass[12pt]{article}
\usepackage[english]{babel}
\usepackage{latexsym}
\usepackage{amssymb}
\usepackage{amsmath}
\usepackage{amsthm}
\usepackage{amsfonts}
\usepackage{graphicx}

\def\'#1{\ifx#1i{\accent"13 \i}\else{\accent"13 #1}\fi}

\newtheorem{theorem}{Theorem}[section]
\newtheorem{corollary}[theorem]{Corollary}

\title{The Hadwiger number, chordal graphs and $ab$-perfection\thanks{Research partially supported by CONACyT-Mexico, Grants 178395, 166306;  PAPIIT-Mexico, Grant IN104915; a Postdoctoral fellowship of CONACyT-Mexico; and the National scholarship programme of the Slovak republic.}}
\author{Christian Rubio-Montiel\\ \emph{christian@matem.unam.mx} \and 
Instituto de Matem{\' a}ticas, 
\\ Universidad Nacional Aut{\' o}noma de M{\' e}xico,
\\ 04510, Mexico City, Mexico \and Department of Algebra,\\ Comenius University,
\\ 84248, Bratislava, Slovakia
}

\providecommand{\keywords}[1]{\textbf{\textit{Keywords:}} #1}

\begin{document}
\maketitle
\begin{abstract}
A graph is \emph{chordal} if every induced cycle has three vertices. The \emph{Hadwiger number} is the order of the largest complete minor of a graph. We characterize the chordal graphs in terms of the Hadwiger number and we also characterize the families of graphs such that for each induced subgraph $H$, (1) the Hadwiger number of $H$ is equal to the maximum clique order of $H$, (2) the Hadwiger number of $H$ is equal to the achromatic number of $H$, (3) the $b$-chromatic  number is equal to the pseudoachromatic number, (4) the pseudo-$b$-chromatic number is equal to the pseudoachromatic number, (5) the Hadwiger number of $H$ is equal to the Grundy number of $H$, and (6) the $b$-chromatic number is equal to the pseudo-Grundy number.
\end{abstract}
\keywords{Complete colorings, perfect graphs, forbidden graphs characterization.}

\textbf{2010 Mathematics Subject Classification:} 05C17; 05C15; 05C83.


\section{Introduction}
Let $G$ be a finite graph. A $k$-\emph{coloring} of $G$ is a surjective function $\varsigma$ that assigns a number from the set $[k]:=\{1,\dots,k\}$ to each vertex of $G$. A $k$-coloring $\varsigma$ of $G$ is called \emph{proper} if any two adjacent vertices have different colors, and $\varsigma$ is called \emph{complete} if for each pair of different colors $i,j\in [k]$ there exists an edge $xy\in E(G)$ such that $x\in \varsigma^{-1}(i)$ and $y\in\varsigma^{-1}(j)$. A $k$-coloring $\varsigma$ of a connected graph $G$ is called \emph{connected} if for all $i\in[k]$, each color class $\varsigma^{-1}(i)$ induces a connected subgraph of $G$.

The \emph{chromatic number $\chi(G)$ of $G$} is the smallest number $k$ for which there exists a proper $k$-coloring of $G$. The \emph{Hadwiger number} $h(G)$ is the maximum $k$ for which a connected and complete coloring of a connected graph $G$ exists, and it is defined as the maximum $h(H)$ among the connected components $H$ of a disconnected graph $G$ (it is also known as \emph{the connected-pseudoachromatic number}, see \cite{MR3265137}).

A graph $H$ is called a \emph{minor} of the graph $G$ if and only if $H$ can be formed from $G$ by deleting edges and vertices and by contracting edges. Suppose that $K_k$ is a minor of a connected graph $G$. If $V(K_k)=[k]$ then there exists a natural corresponding complete $k$-coloring $\varsigma\colon G \rightarrow [k]$ for which $\varsigma^{-1}(i)$ is exactly the set of vertices of $G$ which contract to vertex $i$ in $K_k$. The \emph{Hadwiger number} $h(G)$ of a graph $G$ is the largest $k$ for which $K_k$ is a minor of $G$. Clearly, 

\begin{equation} \label{des1}
\omega(G)\leq h(G)
\end{equation}

where $\omega(G)$ denotes the \emph{clique number} of $G$: the maximum clique order of $G$.

The Hadwiger number was introduced by Hadwiger in 1943 \cite{hadwiger1958ungeloste} together the \emph{Hadwiger conjecture} which states that $\chi(G)\leq h(G)$ for any graph $G$.

The following definition is an extension of the notion of \emph{perfect graph}, introduced by Berge \cite{B}: Let $a,b$ be two distinct parameters of $G$. A graph $ G $ is called \emph{$ab$-perfect} if for every induced subgraph $ H $ of $ G $, $ a(H)=b(H) $. Note that, with this definition a perfect graph is denoted by $ \omega \chi $-perfect. The concept of the $ ab $-perfect graphs was introduced by Christen and Selkow in \cite{MR539075} and extended in \cite{AraujoPardo2013163,AR,MR2954335,MR1801077,montiel2015new,MR1846917}.

A graph $G$ without an induced subgraph $H$ is called $H$-\emph{free}. A graph $H_1$-free, $H_2$-free,... is called \emph{$(H_1,H_2,\dots)$-free}. A \emph{chordal graph} is a $(C_4,C_5,\dots)$-free one.

Some known results are the following: L{\' o}vasz proved in \cite{MR0302480} that a graph $G$ is $\omega\chi$-perfect if and only if its complement is $\omega\chi$-perfect. Chudnovsky, Robertson, Seymour and Thomas proved in \cite{MR2233847} that a graph $ G $ is $ \omega \chi $-perfect if and only if $ G $ and its complement are $(C_5,C_7,\dots)$-free.

This paper is organized as follows: In Section \ref{2} we prove that the families of chordal graphs and the family of $\omega h$-perfect graphs are the same. In Section \ref{3}, we give some consequences of the Section \ref{2} as characterizations of other graph families related to complete colorings.


\section{Chordal graphs and $\omega h$-perfect graphs}\label{2}
We will use the following chordal graph characterization to prove Theorem \ref{teo1}:

\begin{theorem} [Hajnal, Sur{\' a}nyi \cite{MR0103161} and Dirac \cite{MR0130190}] \label{teo0}
A graph $G$ is chordal if and only if $G$ can be obtained by identifying two complete subgraphs of the same order in two chordal graphs.
\end{theorem}

Now, we characterize the chordal graphs and the $\omega h$-perfect ones. The following proof is based on the standard proof of the chordal graph perfection (see \cite{MR2450569}).

\begin{theorem} \label{teo1}
A graph $G$ is $\omega h$-perfect if and only if $G$ is chordal.
\begin{proof}
Assume that $G$ is $\omega h$-perfect. Note that if a cycle $H$ is one of four or more vertices then $\omega(H)=2$ and $h(H)=3 $. Hence, every induced cycle of $G$ has at the most $3$ vertices and the implication is true.

Now, we verify the converse. Since every induced subgraph of a chordal graph is also a chordal graph, it suffices to show that if $G$ is a connected chordal graph, then $\omega(G) = h(G)$. We proceed by induction on the order $n$ of $G$. If $n=1$, then $G=K_1$ and $\omega(G)=h(G)=1$. 
Assume, therefore, that $\omega(H)=h(H)$ for every induced chordal graph $H$ of order less than $n$ for $n\geq 2$ and let $G$ be a chordal graph of order $n$. If $G$ is a complete graph, then $\omega(G)=h(G)=n$. 
Hence, we may assume that $G$ is not complete. By Theorem \ref{teo0}, $G$ can be obtained from two chordal graphs $H_1$ and $H_2$ by identifying two complete subgraphs of the same order in $H_1$ and $H_2$. Let $S$ denote the set of vertices in $G$ that belong to $H_1$ and $H_2$. Thus the induced subgraph $\left\langle S\right\rangle_G $ in $G$ by $S$ is complete and no vertex in $V(H_1)\setminus S$ is adjacent to a vertex in $V(H_2)\setminus S$. Hence,
\[\omega(G)=\max\{\omega(H_1),\omega(H_2)\}=k.\]
Moreover, according to the induction hypothesis, $\omega(H_1)=h(H_1) $ and $\omega(H_2)=h(H_2)$, then 
\[\max\{\omega(H_1),\omega(H_2)\}=\max\{h(H_1),h(H_2)\}=k.\]
On the other hand, since $S$ is a clique cut then each walk between $V(H_1)\setminus S$ and $V(H_2)\setminus S$ contains at least one vertex in $S$. Let $\varsigma$ be a pseudo-connected $h(G)$-coloring of $G$, and suppose there exist two color classes such that one is completely contained in $V(H_1)\setminus S$, and the other one is completely contained in $V(H_2)\setminus S$. Clearly these two color classes do not intersect, which contradicts our choice of $\varsigma$. Moreover, each color class with vertices both in $V(H_1)\setminus S$ and in $V(H_2)\setminus S$, contains vertices in $S$. Consequently, every pair of color classes having vertices both in $V(H_1)\setminus S$ and in $V(H_2)\setminus S$ must have an incidence in $\left\langle S\right\rangle_G$. Thus,
\[h(G)\leq \max\{h(H_1),h(H_2)\}=k.\]
By Equation \ref{des1}, $\omega(G)=k=h(G)$ and the result follows.
\end{proof}
\end{theorem}

It is known that every chordal graph is a $\omega\chi$-perfect one (see \cite{MR2450569}). The following corollary is a consequence of the chordal graph perfection.

\begin{corollary}
Every $\omega h$-perfect graph is $\omega\chi$-perfect.
\end{corollary}


\section{Other classes of $ab$-perfect graphs}\label{3}
In this section, we give a new characterization of several family of $ab$-perfect graphs related to complete colorings. 


\subsection{Achromatic and pseudoachromatic numbers}

Firstly, the \emph{pseudoachromatic number} $\psi(G)$ of $G$ is the largest number $k$ for which there exists a complete $k$-coloring of $G$ \cite{MR0256930}, and it is easy to see that

\begin{equation} \label{des2}
\omega(G)\leq h(G) \leq \psi(G).
\end{equation}

Secondly, the \emph{achromatic number} $\alpha(G)$ of $G$ is the largest number $k$ for which there exists a proper and complete $k$-coloring of $G$ \cite{MR0272662}, and it is not hard to see that

\begin{equation} \label{des3}
\omega(G)\leq \alpha(G) \leq \psi(G).
\end{equation}

Complete bipartite graphs have achromatic number two (see \cite{MR2450569}) but their Hadwiger number can be arbitrarily large, while the graph formed by the union of $K_2$ has Hadwiger number two but its achromatic number can be arbitrarily large. Therefore, $\alpha$ and $h$ are two non comparable parameters.
We will use the following characterization in the proof of Corollary \ref{cor1}.

\begin{theorem} [Araujo-Pardo, R-M \cite{AraujoPardo2013163,AR}] \label{teo2}
A graph $G$ is $\omega\psi$-perfect if and only if $G$ is $(C_4,P_4,P_3 \cup K_2,3K_2) $-free.
\end{theorem}

Corollary \ref{cor1} is an interesting result because it gives a characterization of two non comparable parameters.

\begin{corollary}\label{cor1}
A graph $G$ is $\alpha h$-perfect if and only if $G$ is $\omega\psi$-perfect.
\begin{proof}
Since $h(C_4)=\alpha(P_4)=\alpha(P_3\cup K_2)=\alpha(3K_2)=3$ and $\alpha(C_4)=h(P_4)=h(P_3\cup K_2)=h(3K_2)=2$ (see Figure \ref{fig1}) then a $\alpha h$-perfect graph is $(C_4,P_4,P_3\cup K_2,3K_2)$-free. By Theorem \ref{teo2}, $G$ is $\omega\psi$-perfect.

For the converse, if $G$ is $\omega\psi$-perfect, then by Equation \ref{des2}, $G$ is a $\omega h$-perfect graph, thus, the implication follows.
\end{proof}
\end{corollary}

\begin{corollary}\label{cor4}
Every $\omega\psi$-perfect graph is $\omega\chi$-perfect.
\end{corollary}
\begin{proof}
If a graph $G$ is $\omega\psi$-perfect then Equation (\ref{des2}) implies that $G$ is $\omega h$-perfect, and by Theorem \ref{teo1} $G$ is chordal, therefore $G$ is $\omega\chi$-perfect.
\end{proof}

The following corollary is a consequence of the perfection of $\omega \psi$-perfect graphs.

\begin{corollary}
Every $\alpha h$-perfect graph is $\omega\chi$-perfect.
\end{corollary}


\subsection{$b$-chromatic and pseudo-$b$-chromatic numbers}

On one hand, a coloring such that every color class contains a vertex that has a neighbor in every other color class is called \emph{dominating}. The \emph{pseudo-$b$-chromatic number} $B(G)$ of a graph $G$ is the largest integer $k$ such that $G$ admits a dominating $k$-coloring.

On the other hand, the \emph{$b$-chromatic number} $b(G)$ of $G$ is the largest number $k$ for which there exists a proper and dominating $k$-coloring of $G$ \cite{MR2954335}, therefore

\begin{equation} \label{des5}
\omega(G)\leq b(G) \leq B(G) \leq \psi(G).
\end{equation}

We get the following characterizations:

\begin{corollary}\label{cor3}
For any graph $ G $ the following are equivalent: $\left\langle 1 \right\rangle $ $G$ is $\omega\psi$-perfect, $\left\langle 2 \right\rangle $ $G$ is $b\psi$-perfect, $\left\langle 3 \right\rangle $ $G$ is $B\psi$-perfect and $\left\langle 4 \right\rangle $ $G$ is $(C_4,P_4,P_3 \cup K_2,3K_2)$-free.
\end{corollary}

\begin{proof}
The proofs of $\left\langle 1 \right\rangle \Rightarrow \left\langle 2 \right\rangle$ and $\left\langle 2 \right\rangle \Rightarrow \left\langle 3 \right\rangle$ immediately follow from (\ref{des5}).
To prove $\left\langle 3 \right\rangle \Rightarrow \left\langle 4 \right\rangle$ note that, if $ H \in \{ C_4, P_4, P_3 \cup K_2, 3K_2 \} $ then $ B(H)\not = \psi (H)$, hence the implication is true, see Figure \ref{fig1}.
The proof of $\left\langle 4 \right\rangle \Rightarrow \left\langle 1 \right\rangle$ is a consequence of Theorem \ref{teo2}.
\end{proof}

The following corollary is a consequence of Corollaries \ref{cor4} and \ref{cor3}.

\begin{corollary}
The $b\psi$-perfect graphs and the $B\psi$-perfect ones are $\omega\chi$-perfect.
\end{corollary}

Corollary \ref{cor3} is related to the following theorem:

\begin{theorem}[Christen, Selkow \cite{MR539075} and Blidia, Ikhlef, Maffray \cite{MR2954335}]
For any graph $ G $ the following are equivalent: $\left\langle 1 \right\rangle $ $G$ is $\omega\alpha$-perfect, $\left\langle 2 \right\rangle $ $G$ is $b\alpha$-perfect and $\left\langle 3 \right\rangle $ $G$ is $(P_4,P_3 \cup K_2,3K_2)$-free.
\end{theorem}


\subsection{Grundy and pseudo-Grundy numbers}

First, a coloring of $G$ is called \emph{pseudo-Grundy} if each vertex is adjacent to some vertex of each smaller color. The \emph{pseudo-Grundy number} $\gamma(G)$ is the maximum $k$ for which a pseudo-Grundy $k$-coloring of $G$ exists (see \cite{MR0406799,MR2450569}).

Second, a proper pseudo-Grundy coloring of $G$ is called \emph{Grundy}. The \emph{Grundy number} $\Gamma(G)$ (also known as the \emph{first-fit chromatic number}) is the maximum $k$ for which a Grundy $k$-coloring of $G$ exists (see \cite{MR2450569,G}). From the definitions, we have that

\begin{equation} \label{des4}
\omega(G)\leq\Gamma(G)\leq \gamma(G).
\end{equation}

The following characterization of the graphs call \emph{trivially perfect graphs}, it will be used in the proof of Corollary \ref{cor2}.

\begin{theorem} [R-M \cite{montiel2015new}] \label{teo3}
A graph $G$ is $\omega\gamma$-perfect if and only if $G$ is $(C_4,P_4) $-free.
\end{theorem}

It is known that a trivially perfect graph is chordal (see \cite{MR562306}). The following corollary also gives a characterization of two non comparable parameters.

\begin{corollary}\label{cor2}
A graph $G$ is $\Gamma h$-perfect if and only if $G$ is $\omega\gamma$-perfect.
\begin{proof}

A $\Gamma h$-perfect graph is $(C_4,P_4)$-free because $\Gamma(C_4)=h(P_4)=2$ and $\Gamma(P_4)=h(C_4)=3$ (see Figure \ref{fig1}) then by Theorem \ref{teo3}, $G$ is $\omega\gamma$-perfect.

For the converse, let $G$ be a $\omega\gamma$-perfect graph. If $H$ is an induced graph of $G$, by Equation \ref{des4}, $\omega(H)=\Gamma(H)$. Since $G$ is a chordal graph, $\omega(H)=h(H)$, so the implication follows.
\end{proof}
\end{corollary}
The following corollary is a consequence of the perfection of $\omega \gamma$-perfect graphs.
\begin{corollary}
Every $\Gamma h$-perfect graph is $\omega\chi$-perfect.
\end{corollary}


\subsection{The $b\gamma$-perfect graphs}
Finally, we will use the following characterization of the proof of Theorem \ref{teo5}.

\begin{theorem} [Blidia, Ikhlef, Maffray \cite{MR2954335}] \label{teo4}
A graph $G$ is $b\Gamma$-perfect if and only if $G$ is $(P_4,3P_3,2D)$-free.
\end{theorem}

We get the following characterization.

\begin{theorem} \label{teo5}
A graph $G$ is $b\gamma$-perfect if and only if $G$ is $(C_4,P_4,3P_3,2D)$-free.
\end{theorem}
\begin{proof}
Note that, if $ H \in \{ C_4, P_4, 3P_3,2D \} $ then $ b(H)\not=\gamma(H)$, hence, the implication is true (see Figure \ref{fig1}).

For the converse, a $(C_4,P_4,3P_3,2D)$-free graph $G$ is $\omega\gamma$-perfect (by Theorem \ref{teo3}) and $b\Gamma$-perfect (by Theorem \ref{teo4}). Then, for every induced subgraph $H$ of $G$, $\omega(H)=\gamma(H)=\Gamma(H)$ by Equation \ref{des4} and $b(H)=\Gamma(H)$. Therefore, $b(H)=\gamma(H)$ and the result follows.
\end{proof}
\begin{figure}[!htbp]
\begin{center}
\includegraphics{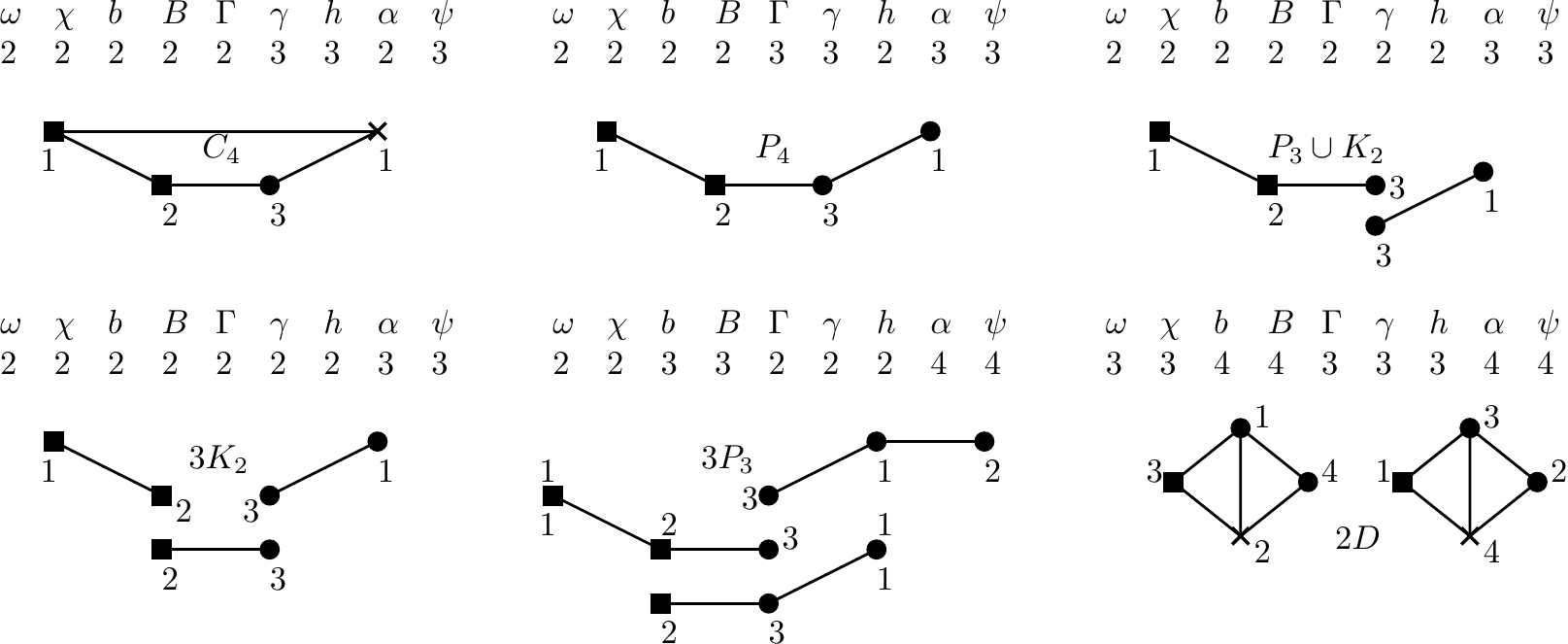}
\caption{\label{fig1} Graphs with a complete coloring with numbers and a connected coloring with symbols.}
\end{center}
\end{figure}


\end{document}